\documentclass[preprint]{elsarticle}

\usepackage{amsfonts,amssymb,amsmath,amsthm}
\usepackage{url}
\usepackage{enumerate}

\usepackage{amsmath,amsfonts,amsthm,url,color,amssymb}
\usepackage{graphicx}

\newcommand{\Z}{\mathbb Z}

\newcommand{\C}{\mathcal C}

\newcommand{\Tr}{\mathrm{Tr}}

\newcommand{\F}{\mathbb{F}}

\newtheorem{theorem}{Theorem}[section]
\newtheorem{lemma}[theorem]{Lemma}

\theoremstyle{definition}
\newtheorem{definition}[theorem]{Definition}

\newtheorem{corollary}[theorem]{Corollary}
\theoremstyle{remark}

\numberwithin{equation}{section}

\begin{document}
\begin{frontmatter}
\title{Counting solutions of special linear equations over finite fields}


%    author one information
%\author[short version for running head]{name for top of paper}
\author{Lucas Reis}
\address{Departamento de Matem\'{a}tica,
Universidade Federal de Minas Gerais,
UFMG,
Belo Horizonte MG (Brazil),
 30270-901}
\ead{lucasreismat@mat.ufmg.br}

%    author two information
%\author{}
%\address{}
%\curraddr{}
%\email{}
%\thanks{}

%    \subjclass is required.

%    "Communicated by" -- provide editor's name; required.

\begin{abstract}
Let $q$ be a prime power, let $\F_q$ be the finite field with $q$ elements and let $d_1, \ldots, d_k$ be positive integers. In this note we explore the number of solutions $(z_1, \ldots, z_k)\in\overline{\F}_q^k$ of the equation
\begin{equation*}L_1(x_1)+\cdots+L_k(x_k)=b,\end{equation*}
with the restrictions $z_i\in \F_{q^{d_i}}$, where each $L_i(x)$ is a non zero polynomial of the form $\sum_{j=0}^{m_i}a_{ij}x^{q^j}\in \F_q[x]$ and $b\in \overline{\F}_q$. We characterize the elements $b$ for which the equation above has a solution and, in affirmative case, we determine the exact number of solutions. As an application of our main result, we obtain the cardinality of the sumset
$$\sum_{i=1}^k\F_{q^{d_i}}:=\{\alpha_1+\cdots+\alpha_k\,|\, \alpha_i\in \F_{q^{d_i}}\}.$$
Our approach also allows us to solve another interesting problem, regarding the existence and number of elements in $\F_{q^n}$ with prescribed traces over intermediate $\F_q$-extensions of $\F_{q^n}$. 
\end{abstract}


\end{frontmatter}

\section{Introduction}
Given an abelian group $(G, +)$, a major problem in additive combinatorics is to study sumsets $A_1+\cdots+A_k:=\{a_1+\cdots+a_k: a_i\in A_k,\; 1\le i\le s\}\subseteq G$, where each $A_i$ is a nonempty subset of $G$. Classical questions include the cardinality (or density if $G$ is infinite countable) of sumsets and a detailed description of their elements. Problems on sumsets relate to many areas such as Combinatorics, Number theory and Algebra. Traditional techniques include, but are not limited to, discrete Fourier analysis, exponential sums, Graph Theory and elementary Number Theory. See~\cite{tao} for a rich source of classical problems, techniques and results  in additive combinatorics and~\cite{charsum} for a survey on results in the finite field setting.  

In this paper we are interested in special sumsets over finite fields from linearized polynomials. More specifically, given the finite field $\F_{q}$ of $q$ elements and $d_1, \ldots, d_k$ positive integers, we study the sumset
$$L_1(\F_{q^{d_1}})+\cdots+L_k(\F_{q^{d_k}}):=\{L_1(a_1)+\cdots+L_k(a_k)\,|\, a_i\in \F_{q^{d_i}},\; 1\le i\le s\},$$
where each $L_i(x)$ is a non zero polynomial of the form 
$\sum_{j=0}^{m_i}a_{ij}x^{q^j}\in \F_q[x]$. 
We actually go a little bit further. If $\overline{\F}_q$ is the algebraic closure of $\F_q$ and $b\in \overline{\F}_q$, we study the number of solutions $(z_1, \ldots, z_k)\in \overline{\F}_q^k$ of the equation
\begin{equation}\label{eq:intro-main}L_1(x_1)+\cdots+L_k(x_k)=b,\end{equation}
with the restrictions $z_i\in \F_{q^{d_i}}$. Of course, the corresponding sumset comprise the elements $b\in \overline{\F}_q$ for which Eq.~\eqref{eq:intro-main} has a positive number of solutions.

The main result of this paper provides a criterion for when a generic $b\in \overline{\F}_q$ yields a solution of Eq.~\eqref{eq:intro-main} with the aforementioned restrictions and, in affirmative case, we determine the exact number of solutions. We observe that, if for a polynomial $g\in \F_q[x]$ with $g(x)=\sum_{i=0}^ma_ix^i$ we set $L_g(x)=\sum_{i=0}^ma_ix^{q^i}$, then each $L_i$ in Eq.~\eqref{eq:intro-main} is written uniquely as $L_{f_i}$ for some nonzero $f_i\in \F_q[x]$. With this observation, our main result can be stated as follows.

\begin{theorem}\label{thm:main}
Fix $f_1, \ldots, f_k\in \F_q[x]$ non zero polynomials, $d_1, \ldots, d_k$ positive integers and $b\in \overline{\F}_q$. Let $\ell$ be any positive integer that is divisible by the numbers $d_1, \ldots, d_k$, set $G_{\ell}(x)=\gcd\left(x^{\ell}-1, \frac{f_1(x)(x^{\ell}-1)}{x^{d_1}-1}, \ldots, \frac{f_k(x)(x^{\ell}-1)}{x^{d_k}-1}\right)$ and $H_{\ell}(x)=\frac{x^{\ell}-1}{G_{\ell}(x)}$. For a polynomial $f\in \F_q[x]$ with $f(x)=\sum_{i=0}^ma_ix^i$, write $L_f(x)=\sum_{i=0}^ma_ix^{q^i}$. Then the equation 
$$L_{f_1}(x_1)+\cdots+L_{f_k}(x_k)=b,$$
has a solution $(z_1, \ldots, z_k)\in \overline{\F}_q^k$ with the restrictions $z_i\in \F_{q^{d_i}}$ if and only if 
$$L_{H_{\ell}}(b)=0.$$
In this case, the number of such solutions equals
$q^{d_1+\cdots+d_k-\deg(H_\ell(x))}.$
\end{theorem}
Theorem~\ref{thm:main} entails the following result in sumsets of finite fields.
\begin{corollary}\label{cor:main}
Fix $k>1$ and $d_1, \ldots, d_k$ positive integers. The number of solutions of
$$x_1+\cdots+x_k=0,$$
with the restrictions $x_i\in \F_{q^{d_i}}$ equals $q^{d_1+\cdots+d_k-\lambda(d_1, \ldots, d_k)}$, where
 $$\lambda(d_1, \ldots, d_k):=\sum_{i=1}^kd_i+\sum_{j=2}^{k}(-1)^{j+1}\sum_{1\le i_1<\cdots<i_j\le k}\gcd(d_{i_1}, \ldots, d_{i_j}).$$ 
In particular, the sumset 
$\sum_{i=1}^k\F_{q^{d_i}}:=\{\alpha_1+\cdots+\alpha_k\,|\, \alpha_i\in \F_{q^{d_i}}\}$ contains exactly $q^{\lambda(d_1, \ldots, d_k)}$ elements. 
\end{corollary}

In contrast to many sophisticated techniques that are traditionally employed when studying sumsets, our setting allows us to employ a very elementary approach. Our method relies on considering the generating property of {\em normal elements} over finite fields, the linear structure of linearized polynomials $\sum_{j=0}^ma_jx^{q^j}$ and the arithmetic of the polynomial ring $\F_q[x]$. This approach also allows us to obtain results on the existence and number of elements in finite field extensions with prescribed traces over intermediate extensions.

The structure of the paper is given as follows. Section 2 provides some definitions and preliminary results that are used along the way and in Section 3 we prove Theorem~\ref{thm:main}. Finally, in Section 4 we discuss the existence of elements in finite field extensions with several prescribed traces.

\section{Preparation}
Throughout this paper, $q$ is a prime power, $\F_q$ is the finite field with $q$ elements and $\overline{\F}_q$ is the algebraic closure of $\F_q$. We recall useful definitions in the theory of finite fields and provide some technical results that are frequently employed. For a more detailed information on these definitions and results, see Sections 2.3 and 3.4 of~\cite{LN}.

\begin{definition}
\begin{enumerate}
\item For a polynomial $f\in \F_q[x]$ with $f(x)=\sum_{i=0}^ma_ix^i$, the polynomial $L_f(x):=\sum_{i=0}^ma_ix^{q^i}$ is the {\em $q$-associate} of $f(x)$.

\item Fix $n$ a positive integer. An element $\beta\in \F_{q^n}$ is {\em normal} over $\F_q$ if the elements $\beta, \beta^q,\ldots, \beta^{q^{n-1}}$ form a basis for $\F_{q^n}$ as an $\F_q$-vector space.
\end{enumerate}
\end{definition}

The existence of normal elements is known for any extension $\F_{q^n}$ of $\F_q$. The following lemma provides some basic properties of the $q$-associate of polynomials over $\F_q$. Its proof is straightforward so we omit details.

\begin{lemma}\label{lem:q-associate}
For any $f, g\in \F_q[x]$, we have that $L_{f+g}(x)=L_f(x)+L_g(x)$ and $L_{f}(L_g(x))=L_{f\cdot g}(x)$. In particular, if $f$ divides $g$, then $L_f$ divides $L_g$.
\end{lemma}

We emphasize that Lemma~\ref{lem:q-associate} is frequently used along the way. Let $\mathcal C(n)$ be the set of polynomials over $\F_q$ that are either constant or have degree at most $n-1$. We observe that $\mathcal C(n)$ is an $n$-dimensional $\F_q$-vector space and can be seen as the simplest presentation of the quotient ring $\frac{\F_q[x]}{x^n-1}$; this fact is extensively employed. Normal elements can be used to generate finite fields using the sets $\mathcal C(n)$, as follows.

\begin{lemma}\label{lem:generator}
Fix $n$ a positive integer and $\beta\in \F_{q^n}$ a normal element. Then any $\alpha\in \F_{q^n}$ is written uniquely as $L_f(\beta)$ for some $f\in \C(n)$.
\end{lemma}

\begin{proof}
This follows directly by the fact that $\{\beta, \beta^q, \ldots, \beta^{q^{n-1}}\}$ is an $\F_q$-basis for $\F_{q^n}$.
\end{proof}

We observe that, for $\alpha\in \overline{\F}_{q}$, $L_{x^n-1}(\alpha)=0$ if and only if $\alpha^{q^n}-\alpha=0$, i.e., $\alpha\in \F_{q^n}$. In particular, from Lemma~\ref{lem:q-associate}, the set $\mathcal I_{\alpha}:=\{h\in \F_q[x]\,|\, L_h(\alpha)=0\}$ is a non zero ideal of $\F_q[x]$. Since $\F_q[x]$ is a Principal Ideal Domain, $\mathcal I_{\alpha}$ is generated by a monic polynomial, uniquely determined by $\alpha$.

\begin{definition}
For $\alpha\in \overline{\F}_{q}$, let $m_{\alpha, q}\in \F_q[x]$ be the monic polynomial that generates the ideal $\mathcal I_{\alpha}$.
\end{definition}

In particular, from the definition of $m_{\alpha, q}(x)$, $\alpha\in \F_{q^n}$ if and only if $m_{\alpha, q}(x)$ divides $x^n-1$. From this observation and Lemma~\ref{lem:generator}, the following corollary is straightforward.

\begin{corollary}
An element $\beta\in \F_{q^n}$ is normal over $\F_q$ if and only if $m_{\beta, q}(x)=x^n-1$.
\end{corollary}

The following lemma relates the polynomials $m_{\alpha, q}$ and $m_{\beta, q}$, where $\alpha$ is the image of $\beta$ by a linearized polynomial.

\begin{lemma}\label{lem:gcd}
Let $\beta\in \overline{\F}_q$ and fix $g\in \F_q[x]$. If $\alpha=L_g(\beta)$, then
$$m_{\alpha, q}=\frac{m_{\beta, q}}{\gcd(m_{\beta, q}, g)}.$$
\end{lemma}
\begin{proof}
Set $h=\frac{m_{\beta, q}}{\gcd(m_{\beta, q}, g)}$, hence $g\cdot h$ is divisible by $m_{\beta, q}$. Therefore, $L_h(\alpha)=L_{g\cdot h}(\beta)=0$ and so $m_{\alpha, q}$ divides $h$. If it divides strictly, there exists a nontrivial divisor $f$ of $h$ such that $L_{\frac{h}{f}}(\alpha)=0$, hence $L_H(\beta)=0$, where $H=\frac{gh}{f}$. In particular, $m_{\beta, q}$ divides $H$ or, equivalently, $f$ divides $\frac{g}{\gcd(m_{\beta, q}, g)}$. However, from construction, the polynomials $\frac{g}{\gcd(m_{\beta, q}, g)}$ and $h$ are relatively prime and so they cannot have $f$ as a common divisor.

\end{proof}

From Lemmas~\ref{lem:gcd} and~\ref{lem:generator}, we obtain the following corollary.

\begin{corollary}\label{cor:char}
Fix $n$ a positive integer, $m$ a divisor of $n$ and $\beta\in \F_{q^n}$ a normal element. Then the elements of $\F_{q^m}$ are written uniquely as $L_{\frac{x^n-1}{x^m-1}\cdot h(x)}(\beta)$ with $h\in \C(m)$.
\end{corollary}

\begin{proof}
Fix $\alpha\in \F_{q^m}$. From Lemma~\ref{lem:generator}, there exists a unique $H\in \C(n)$ such that $\alpha=L_{H}(\beta)$. Since $\beta$ is normal, $m_{\beta, q}(x)=x^n-1$. From Lemma~\ref{lem:gcd}, $m_{\alpha, q}(x)=\frac{x^n-1}{\gcd(x^n-1, H(x))}$. But $\alpha\in \F_{q^m}$, hence $m_{\alpha, q}(x)$ divides $x^m-1$, i.e., $\gcd(x^n-1, H(x))$ is divisible by $\frac{x^n-1}{x^m-1}$. The latter implies that $H(x)= \frac{x^n-1}{x^m-1}\cdot h(x)$ for some $h\in \F_q[x]$. Since $H$ is unique and $H\in \C(n)$, we have that $h$ is unique and $h\in \C(m)$.
\end{proof}

The following theorem provides the main auxiliary result of this paper.

\begin{theorem}\label{thm:aux}
Fix $n$ a positive integer, $f\in \F_q[x]$ a divisor of $x^n-1$ and $\alpha\in \overline{\F}_q$. If $\beta\in \F_{q^n}$ is a normal element, then there exists a polynomial $g\in \F_q[x]$ such that $L_{f\cdot g}(\beta)=\alpha$ if and only if $m_{\alpha, q}(x)$ divides $\frac{x^n-1}{f(x)}$ or, equivalently, $L_{\frac{x^n-1}{f(x)}}(\alpha)=0$.
\end{theorem}

\begin{proof}
The ``only if part'' follows directly from Lemma~\ref{lem:gcd}. For the ``if'' part, suppose that $m_{\alpha, q}(x)$ divides $\ell(x)=\frac{x^n-1}{f(x)}$ and fix $\beta\in \F_{q^n}$ a normal element. From Lemma~\ref{lem:generator}, there exists $F\in \C(n)$ such that $\alpha=L_F(\beta)$ and, from Lemma~\ref{lem:gcd}, $m_{\alpha, q}(x)=\frac{x^n-1}{\gcd(x^n-1, F(x))}$. Therefore, $f(x)$ divides $\gcd(x^n-1, F(x))$, i.e., $F=fg$ for some $g\in \F_q[x]$. In other words,
$\alpha=L_{f\cdot g}(\beta)$.
\end{proof}

\section{Proof of the main result}
Fix $b\in \overline{\F}_q$, $d_1, \ldots, d_k$ positive integers and $f_1, \ldots, f_k\in\F_q[x]$ non zero polynomials. Let $\ell>0$ be any integer divisible by $d_1, \ldots, d_k$ and fix $\beta\in \F_{q^{\ell}}$ an element that is normal over $\F_q$. In particular $\F_{q^{d_i}}\subseteq \F_{q^{\ell}}$ for any $1\le i\le k$. For each $(z_1, \ldots, z_k)\in \prod_{i=1}^k\F_{q^{d_i}}$, Corollary~\ref{cor:char} entails that there exists a unique $k$-tuple $(h_1, \ldots, h_k)\in\prod_{i=1}^k\C(d_i)$ such that $z_i=L_{h_i(x)\frac{x^{\ell}-1}{x^{d_i}-1}}(\beta)$. In particular, if we set $\varphi_{\beta}: \prod_{i=1}^k\F_{q^{d_i}}\to \F_{q^{\ell}}$ with
$$(z_1, \ldots, z_k)\mapsto L_{f_1}(z_1)+\cdots+L_{f_k}(z_k),$$
$\varphi_{\beta}$ is an $\F_q$-linear map between $\F_q$-vector spaces with image set $\mathcal S_{\beta}=\{L_{h}(\beta)\,|\, h\in \mathcal S\}$, where
$$\mathcal S=\left\{ \sum_{i=1}^k\frac{f_i(x)(x^{\ell}-1)}{x^{d_i}-1}\cdot h_i(x)\,|\, h_i\in \C(d_i)\right\}.$$
Since $L_{x^{\ell}-1}(\beta)=\beta^{q^{\ell}}-\beta=0$, we have that $\mathcal S_{\beta}=\{L_{h}(\beta)\,|\, h\in \mathcal S'\}$, where 
$$\mathcal S'=\left\{ h(x)\cdot (x^{\ell}-1)+\sum_{i=1}^k\frac{f_i(x)(x^{\ell}-1)}{x^{d_i}-1}\cdot h_i(x)\,|\,h,  h_i\in \F_q[x]\right\},$$
is the ideal of $\F_q[x]$ generated by the polynomials $\frac{f_i(x)(x^{\ell}-1)}{x^{d_i}-1}$ and the polynomial $x^{\ell}-1$. Since $\F_q[x]$ is an Euclidean Domain, the ideal $\mathcal S'$ is generated by the polynomial 
$$G_{\ell}(x)=\gcd\left(x^{\ell}-1, \frac{f_1(x)(x^{\ell}-1)}{x^{d_1}-1}, \ldots, \frac{f_k(x)(x^{\ell}-1)}{x^{d_k}-1}\right).$$
In particular, $\mathcal S_{\beta}=\{L_{h}(\beta)\,|\, h\in \mathcal S^*\}$, where $$\mathcal S^*=\{g(x)\cdot G_{\ell}(x)\,|\, g\in \C(\ell-\deg(G_{\ell(x)}))\}.$$
Therefore, the equation 
$$L_{f_1}(x_1)+\cdots+L_{f_k}(x_k)=b,$$
has a solution $(z_1, \ldots, z_k)\in \overline{\F}_q^k$ with the restrictions $z_i\in \F_{q^{d_i}}$ if and only if $b$ is of the form $L_{g\cdot G_{\ell}}(\beta)$ for some $g\in\F_{q}[x]$. From Theorem~\ref{thm:aux}, the later holds if and only if $L_{H_{\ell}}(b)=0$, where
$H_{\ell}(x)=\frac{x^{\ell}-1}{G_{\ell}(x)}$. 

For the number of solutions, we observe that such number coincides with the number of solutions for $b=0$. In this case, we are just looking at the kernel of $\varphi_{\beta}$, which is an $\F_{q}$-vector space so it suffices to compute its dimension.
We have seen that the image set of $\varphi_{\beta}$ comprises the roots of $L_{H_{\ell}}(y)=0$ that lie in $\F_{q^{\ell}}$. Since $H_{\ell}$ divides $x^{\ell}-1$, Lemma~\ref{lem:q-associate} entails that $L_{H_{\ell}}(x)$ divides $x^{q^{\ell}}-x$, a separable polynomial. Therefore, all the $q^{\deg(H_\ell(x))}$ roots of $L_{H_{\ell}}(y)=0$ are distinct and lie in $\F_{q^{\ell}}$. In particular, the image set of $\varphi_{\beta}$ is an $\F_q$-vector space of dimension $\deg(H_{\ell}(x))$. The Rank-Nullity Theorem entails that the kernel of $\varphi_{\beta}$ is an $\F_q$-vector space with dimension
$\sum_{i=1}^{k}d_i-\deg(H_{\ell}(x))$,
hence it has cardinality $q^{d_1+\cdots+d_k-\deg(H_{\ell}(x))}$.

\subsection{Proof of Corollary~\ref{cor:main}}
Let $\ell$ be the least common multiple of the numbers $d_1, \ldots, d_k$. We observe that Theorem~\ref{thm:main} applies to  Corollary~\ref{cor:main} by setting $f_i(x)=1\in \F_q$. In particular, the number of solutions of
\begin{equation}\label{eq:zero-sum}x_1+\cdots+x_k=0,\end{equation}
with the restrictions $x_i\in \F_{q^{d_i}}$ equals
$q^{d_1+\cdots+d_k-\deg(H_{\ell}(x))}$,
where $$H_{\ell}(x)=(x^{\ell}-1)\cdot \gcd\left(x^{\ell}-1, \frac{(x^{\ell}-1)}{x^{d_1}-1}, \ldots, \frac{(x^{\ell}-1)}{x^{d_k}-1}\right)^{-1}=\mathrm{lcm}(x^{d_1}-1, \ldots, x^{d_k}-1).$$

We observe that $\gcd(x^{a}-1, x^{b}-1)=x^{\gcd(a, b)}-1$ and $\mathrm{lcm}(x^{a}-1, x^b-1)=\frac{(x^a-1)(x^b-1)}{\gcd(x^a-1, x^b-1)}$ for any positive integers $a, b$. By a simple inclusion-exclusion argument, the latter implies that $\mathrm{lcm}(x^{d_1}-1, \ldots, x^{d_k}-1)$ is a polynomial of degree 
$$\lambda(d_1, \ldots, d_k)=\sum_{i=1}^kd_i+\sum_{j=2}^{k}(-1)^{j+1}\sum_{1\le i_1<\cdots<i_j\le k}\gcd(d_{i_1}, \ldots, d_{i_j}),$$
and so the number of solutions of Eq.~\eqref{eq:zero-sum} equals $q^{d_1+\cdots+d_k-\lambda(d_1, \ldots, d_k)}$. We proceed to the formula regarding the cardinality of the sumset $$\mathfrak S=\sum_{i=1}^k\F_{q^{d_i}}:=\{\alpha_1+\cdots+\alpha_k\,|\, \alpha_i\in \F_{q^{d_i}}\}.$$ This follows directly from the previous enumeration formula and the Rank-Nullity Theorem for the $\F_q$-linear map
$\varphi:\prod_{i=1}^k\F_{q^i}\to \mathfrak S$ given by $$(\alpha_1, \ldots, \alpha_k)\mapsto \sum_{i=1}^k\alpha_i.$$

\section{On elements with several prescribed traces}
In this section, we consider special  systems of linear equations over finite fields involving trace maps. We recall that, for a positive integer $n>1$ and a divisor $m<n$ of $n$, the trace map $\mathrm{Tr}_{n/m}:\F_{q^n}\to \F_{q^m}$ is given by $$\alpha\mapsto \sum_{j=0}^{n/m-1}\alpha^{q^{mi}}=L_{\frac{x^n-1}{x^m-1}}(\alpha).$$

%It is well-known that $\mathrm{Tr}_{n/m}$ is {\em onto}; this fact will be frequently used without mention. 

In~\cite{count} the authors provide the number of elements in $\F_{q^n}$, not contained in any intermediate extension $\F_{q^d}$, with prescribed trace $a\in \F_q^*$. Given a $k$-tuple $(d_1, \ldots, d_k)$ of distinct divisors of $n$ and $(\beta_1, \ldots, \beta_k)\in \prod_{i=1}^k\F_{q^{k_i}}$, we may ask if there is any element $\alpha\in \F_{q^n}$ such that $\mathrm{Tr}_{n/d_i}(\alpha)=\beta_i$. In other words, we are looking at solutions to the following system of equations:
$$L_{\frac{x^n-1}{x^{d_i}-1}}(x)=\beta_i, 1\le i\le k.$$ 
It is direct to verify that the trace map is transitive, i.e., $\mathrm{Tr}_{n/l}(x)=\mathrm{Tr}_{m/l}(\mathrm{Tr}_{n/m}(x))$ whenever $m$ divides $n$ and $l$ divides $m$. In particular, if the system of equations has a solution $z\in \F_{q^n}$, then \begin{equation}\label{eq:trace-pair}\Tr_{d_i/\gcd(d_i, d_j)}(\beta_i)=\Tr_{d_j/\gcd(d_i, d_j)}(\beta_j), 1\le i, j\le k.\end{equation} 
Therefore, a necessary condition for the system to have a solution is that Eq.~\eqref{eq:trace-pair} holds. In the following theorem, we show that such condition is also sufficient and we obtain the exactly number of solutions.
\begin{theorem}\label{thm:trace}
Let $n>1$ be a positive integer and let $d_1, \ldots, d_k$ be distinct divisors of $n$. Fix $(\beta_1, \ldots, \beta_k)\in \prod_{i=1}^k\F_{q^{k_i}}$ and consider the following system of equations:
\begin{equation}\label{eq:trace-main}\Tr_{n/d_i}(x)=\beta_i,\; 1\le i\le k,
\end{equation}
where $\Tr_{d_i/\gcd(d_i, d_j)}(\beta_i)=\Tr_{d_j/\gcd(d_i, d_j)}(\beta_j)$ for any $1\le i<j\le k$.
Then Eq.~\eqref{eq:trace-main} has exactly $q^{n-\lambda(d_1, \ldots, d_k)}$ solutions, where 
$$\lambda(d_1, \ldots, d_k)=\sum_{i=1}^kd_i+\sum_{j=2}^{k}(-1)^{j+1}\sum_{1\le i_1<\cdots<i_j\le k}\gcd(d_{i_1}, \ldots, d_{i_j}).$$
\end{theorem}

\begin{proof}Fix $\beta\in \F_{q^n}$ a normal element. Since $\beta_i\in \F_{q^{d_i}}$, Corollary~\ref{cor:char} entails that there exists unique $h_i\in \C(d_i)$ such that $\beta_i=L_{\frac{x^n-1}{x^{d_i}-1}\cdot h_i(x)}(\beta)$. Moreover, an arbitrary element of $\F_{q^n}$ is of the form $L_F(\beta)$ with $F \in \C(n)$. In particular, since $L_T(\beta)=0$ if and only if $T(x)\equiv 0\pmod {x^n-1}$, Eq.~\eqref{eq:trace-main} is equivalent to the following system of congruences in $\F_{q}[x]$:
\begin{equation*}\frac{x^n-1}{x^{d_i}-1}\cdot F(x)\equiv \frac{x^n-1}{x^{d_i}-1}\cdot h_i(x)\pmod {x^n-1},\; 1\le i\le k.
\end{equation*}
The previous system is equivalent to 
\begin{equation}\label{eq:trace-main-revised}
F(x)\equiv h_i(x)\pmod {x^{d_i}-1},\; 1\le i\le k.
\end{equation}
From hypothesis, $\Tr_{d_i/\gcd(d_i, d_j)}(\beta_i)=\Tr_{d_i/\gcd(d_i, d_j)}(\beta_j)$ for any $1\le i<j\le k$. The latter can be rewritten as
$$h_i(x)\equiv h_j(x)\pmod {x^{\gcd(d_i, d_j)}-1}, \; 1\le i<j\le k.$$
From the previous equalities, the Chinese Remainder Theorem entails that Eq.~\eqref{eq:trace-main-revised} has exactly one solution $F_0\in \C(\deg(L))$, where $L(x):=\mathrm{lcm}(x^{d_1}-1, \ldots, x^{d_k}-1)$. Moreover, any other solution is of the form $F_0(x)+L(x)\cdot M(x)$ with $M\in \F_q[x]$. We have seen that $L(x)$ has degree $\lambda(d_1, \ldots, d_k)$. We recall that the only restriction on Eq.~\eqref{eq:trace-main-revised} is $F\in \C(n)$, i.e., $M\in \C(n-\deg(L))$. The latter implies that we have exactly $q^{n-\lambda(d_1, \ldots, d_k)}$ distinct solutions for Eq.~\eqref{eq:trace-main}.

\end{proof}

\section*{Acknowledgements}
We thank the anonymous reviewers for helpful comments and suggestions. The author was supported by FAPESP 2018/03038-2, Brazil.

%%%%%%%%%%% To ease editing, use normal size for the references:


\begin{thebibliography}{}

%\bibitem{HKT} J.W.P.~Hirschfeld, G.~Korchm\'aros and F.~Torres,
%{\em Algebraic curves over a finite field}, 
%Princeton Univ. Press, 2008.

%\bibitem{wfr}
%Y.D.~Karabulut, Waring's problem in finite rings
%{\em J. Pure Appl. Algebra} 223(8): 3318--3329, 2019.

\bibitem{LN} R.~Lidl and H.~Niederreiter,
{\em Introduction to finite fields and their applications}, 
Cambridge University Press New York, NY, USA 1986.

\bibitem{count}  F. Ruskey, C. R. Miers and J. Sawada, The number of irreducible polynomials and Lyndon words with given trace, {\em SIAM J. Discrete Math.} 14: 240--245, 2001.

\bibitem{charsum} I.~Shparlinski,  Additive Combinatorics over Finite Fields: New Results and Applications.
\newblock in P.~Charpin, A.~Pott and A.~Winterhof. \newblock {\em Finite Fields and Their Applications - Character Sums and Polynomials}.
\newblock De Grutyer, Radon Series on Computational and applied mathematics (11), 2013.


\bibitem{tao} T.~Tao and V.~Vu, {\em Additive Combinatorics}  (Cambridge Studies in Advanced Mathematics), Cambridge: Cambridge University Press (2006).


%\bibitem{VR} R.~Vaughan and T.~Wooley, {\em Waring’s problem: a survey}. In: Number theory for the millennium, vol. 3, 301–-340. A K Peters, Urbana (2000).

%\bibitem{W}  A.~Weil,
%\emph{Numbers of solutions of equations in finite fields},
% Bull. Amer. Math. Soc. 55 (1949), 497-508.
\end{thebibliography}
\end{document}